\documentclass[conf]{new-aiaa}    
\usepackage[utf8]{inputenc}

\usepackage{graphicx}
\usepackage{amsmath}
\usepackage[version=4]{mhchem}
\usepackage{siunitx}
\usepackage{longtable,tabularx}
\usepackage{amsthm}

\usepackage{subfigure}

\newcommand{\mbf}[1]{\mathbf{#1}}              
\newcommand{\mbs}[1]{{\boldsymbol{#1}}}

\newcommand{\trans}{{\ensuremath{\mathsf{T}}}}  
\newcommand{\trace}{ {\ensuremath{\mathrm{tr}}} }   

\newcommand{\norm}[1]{\left\Vert#1\right\Vert} 

\newcommand{\bbm}{\begin{bmatrix}}
\newcommand{\ebm}{\end{bmatrix}}

\newtheorem{theorem}{Theorem}[section]

\newcommand{\onehalf}{\mbox{$\textstyle{\frac{1}{2}}$}}
\newcommand{\pd}[2]{{\frac{\partial #1}{\partial #2}}} 
\newcommand{\bone}{\mbf{1}}
\newcommand{\bzero}{\mbf{0}}

\newcommand{\mbbRv}[1]{\mathbb{R}^{#1}}

\newcommand{\mbfx}{\mbf{x}}
\newcommand{\mbfy}{\mbf{y}}

\newcommand{\mbfxx}{\mbf{X}}
\newcommand{\mbfyy}{\mbf{Y}}
\newcommand{\mbfz}{\mbf{Z}}

\newcommand{\mbfe}{\mbf{e}}
\newcommand{\mbfv}{\mbf{v}}

\newcommand{\mbfu}{\mbf{u}}
\newcommand{\mbfp}{\mbf{p}}

\newcommand{\mbfh}{\mbf{H}_k}
\newcommand{\mbfb}{\mbf{B}}
\newcommand{\mbfpp}{\mbf{P}}
\newcommand{\mbfk}{\mbf{K}_k}

\newcommand{\dx}{\delta \mbf{x}}
\newcommand{\dy}{\delta \mbf{y}}
\newcommand{\ex}[1]{{\mathbb{E}\left(#1\right)}}

\title{Bounding the Estimation Error Covariance for Nonlinear Systems} 

\begin{document}

\author{Sze Kwan Cheah\footnote{Ph.D. Student, Department of Aerospace Engineering and Mechanics, email: coffeebean@gmail.com.}}
\author{Yingjie Hu\footnote{Ph.D. Student, Department of Aerospace Engineering and Mechanics, email: hu000258@umn.edu.}}

\affil{University of Minnesota, Minneapolis, MN 55455}

\maketitle

\begin{abstract} 

This paper presents preliminary work on computing upper bounds on the estimation error covariance in the framework of the extended Kalman filter. 
The approach taken is using quadratic constraints to bound the dynamic nonlinearities and use of semidefinite programs to find the upper bound of each entry of the estimation error covariance matrix. 

\end{abstract}

\section{Introduction}
Among the most popular techniques for state estimation for nonlinear systems is the Extended Kalman Filter (EKF) \cite{kalman1960new}. 
The EKF may be viewed as composed of two parts.
The first part is the Kalman filter time update or also commonly known as the prediction step that propagates the state estimate as well as the state estimation error covariance. 
The second part is composed of the Kalman filter measurement update or also commonly known as the correction step. Here, the minimum mean-square error estimator (Kalman gain), is computed and is used to optimally weigh the measurement and current state estimate to produce an updated state estimate and corresponding updated estimation error covariance. More details can be found in~\cite{simon2006optimal}.
Within the time update/prediction step, the EKF makes use of a linearization of the nonlinear system's process model to propagate the estimation error covariance. The effect of the system's nonlinearities is not captured, which can result in an unreliable probabilistic description of the estimation error that is unsuitable for safety-critical applications, such as aerospace systems. 

Incorporating information regarding the nonlinearities of a process model in the time-update propagation of the estimation error covariance
has been studied extensively, with methods that include a second-order EKF \cite{gelb1974applied} that attempts to reduce the linearization errors of the EKF using a second-order Taylor series expansion. 
Even though this increase in order improves the estimates, it does not completely capture or bound the effect of the nonlinearties. 
In the unscented Kalman filter \cite{simon2006optimal}, nonlinear transformations on multiple points on the probability distribution function is used to compute estimates of the states and state error covariance. 
While capturing the nonlinearity can provide good estimates of the estimation error covariance, this method does not provide conservative bounds on the computed covariance that may be needed for safety-critical applications. 

Instead of simply improving the estimates of the estimation error covariance, an alternative approach is to find an upper bound of the estimation error covariance matrix.
A commonly accepted means of bounding the covariance matrix is with the Loewner order where the difference between the bounding matrix and the true covariance matrix is positive semidefinite, as discussed in~\cite{noack2017decentralized}. 
Furthermore, \cite{luboldClarkTaylor2022formal} provides 
an overview of the literature on covariance overbounding using a relative definiteness measure of conservatism between the true covariance matrix and its bound. 

Separately, the trace of the estimation error covariance is commonly used as a function to be minimized when considering the optimality of sensor fusion algorithms. 
%
A conservative linear unbiased estimator is proposed in
\cite{forsling2022conservative} by solving a semidefinite program (SDP) that minimizes the covariance trace to estimate the lower and upper bound of the state estimation error covariance.
Similarly, the authors of \cite{gao2015robust} fashioned a linear state estimation algorithm based on the minimax problem for the trace of the state estimation error covariance. 

Propagation of the state estimation error covariance using the process model takes advantage of the system knowledge as the time update equation in the traditional EKF. When used in conjunction with state estimation from measurements, the fused state estimate and the corresponding state estimation error is used for decision making. 
In \cite{li2016new}, the authors 
found the upper and lower bounds of the estimation error covariance matrix but did not consider the possibility of nonlinearities. 
Bounding of the propagated covariance has been studied \cite{yang2002robust} in the context of time-varying uncertain system that is norm-bounded by solving discrete Riccati difference equations.


One important requirement of filter design for safety-critical navigation systems is that it can maintain bounded error uncertainties. In traditional EKF, the error covariance is propagated forward using the linearized dynamic equation of the system. However, this approach does not properly capture the impact of nonlinearities on the covariance matrix and can potentially lead to underestimated error uncertainties, which can cause hazardous consequences, especially for safety-critical navigation. To address this issue, we propose to develop bounds on the error covariance estimates for nonlinear system within the EKF framework. 

This paper provides a preliminary theoretical work and associated algorithm to bound individual entries of the state error covariance matrix that is typically computed in the framework of the EKF. Given that both the lower and upper bound of each element of the matrix is available, it is speculated that an upper bound matrix by definition of positive semidefinite definition of conservative bounds can be obtained. 

\textit{Notation:} 
%
We use $>0$ and $\geq 0$ respectively to denote the positive definiteness and positive semidefiniteness of a matrix.
Similarly, the inequality of $< 0$ and $\leq 0$ respectively denotes negative definite and negative semidefinite matrices. Repeated blocks within symmetric matrices are replaced by $*$ for brevity and clarity.
For simplicity, the identity matrix is written as $\mbf{1}$ and a matrix of zeros is written as $\bzero$.
The basis vector of $\bone_i$ is a vector of zeros with the $i^{th}$ entry of 1. 

\section{Problem Statement}
Consider the nonlinear discrete-time system dynamics of 
\begin{align}
    \mbf{x}_{k+1} &= \mbf{f}_d (\mbf{x}_k, \mbf{u}_k, \mbf{w}_k)    \label{eq:dt_eom0} , \\
    \mbf{y}_{k} &= \mbf{g}_d (\mbf{x}_k, \mbf{v}_k) \label{eq:meas}. 
\end{align}
where $k$ denotes the discrete time step, $\mbfx \in \mathbb{R}^{n_x}$ the state variable, $\mbf{y} \in \mathbb{R}^{n_y}$ the measurement variable and $\mbf{u} \in \mathbb{R}^{n_u}$ the known input. The process noise and measurement noise are respectively $\mbf{w}_k \in \mathbb{R}^{n_w}$ and $\mbf{v}_k \in \mathbb{R}^{n_v}$. Both are assumed zero-mean Gaussian distributed. Their associated covariance is defined as $\mbf{Q} = \mathbb{E}(\mbf{w}_k \mbf{w}_k^\trans)$  and $\mbf{R}=\ex{\mbf{v}_k \mbf{v}_k^\trans}$. 

The EKF \cite{simon2006optimal} uses a linearized time update equation \eqref{eq:dt_eom0} to propagate the state error covariance $\mbfpp^+_k=\mathbb{E}(\delta \mbf{x}_{k} \delta \mbf{x}_{k}^\trans)$ from time step $k$ to the next time step $\mbfpp^-_{k+1}=\mathbb{E}(\delta \mbf{x}_{k+1} \delta \mbf{x}_{k+1}^\trans)$. It further utilizes a linearized measurement equation \eqref{eq:meas} to obtain the \textit{a posteriori} covariance estimate $\mbfpp^+_{k+1}$ using a computed Kalman gain $\mbfk$. 
%
%
Instead of linearization the process and measurement equation, they are first decomposed to the linear and nonlinear parts for future analysis. 

Defining the \textit{a priori estimate} at time $k$ as $\check{\mbf{x}}_k$, the estimation error time update equation can be expressed as 
\begin{align}
    \delta \mbf{x}_{k+1} &= \mbf{f}_d (\mbf{x}_k, \mbf{u}_k, \mbf{w}_k) - 
    \mbf{f}_d (\check{\mbf{x}}_k, \mbf{u}_k, \bzero). 
    \label{eq:dt_eom}
\end{align}
A decomposition can be performed, without loss of generality, at each time step of \eqref{eq:dt_eom} as 
\begin{align}
    \delta \mbf{x}_{k+1} &= \mbf{A}_k \delta \mbf{x}_k + \mbf{B}_p \mbf{p}_k + \mbf{B}_w \mbf{w}_k, 
    \label{eq:eom_dt_1} \\
    \mbf{p}_k &= \grave{\mbs{\Delta}}_k (\mbs{\theta}_k),
    \label{eq:eom_dt_2} \\
    \mbs{\theta}_k &= \mbf{C}_\theta \delta \mbf{x}_k,  \label{eq:eom_dt_last}
\end{align}
where $\left.\mbf{A}_k=\pd{\mbf{f}}{\mbf{x}}\right|_{(\check{\mbf{x}}_k,\mbf{u}_k)}$ varies with both state and input. All matrices including $\mbf{B}_p , \mbf{B}_w$ and $\mbf{C}_\theta$, as well as the function $\grave{\mbs{\Delta}}_k$ change at every time index to reflect \eqref{eq:dt_eom}. 

%

%
Similarly, the nonlinear measurement equation \eqref{eq:meas} may be equivalently decomposed without approximation as
\begin{align}
    \dy_k &= \mbfy - \check{\mbfy} = \mbfh \dx_k + \mbfb_\rho \mbs{\rho}_k + \mbfb_v \mbfv_k , \label{eq:meas_comp} \\
    \mbs{\rho}_k &= \acute{\mbs{\Delta}}_k (\mbs{\mu}) , \label{eq:delta} \\
    \mbs{\mu} &= \mbf{C}_\mu \dx_k. \label{eq:meas_end}
\end{align}
Note that $\dy$ is a decomposition of  $\dy_k = \mbf{g}(\mbfx_k, \mbfu_k)+ \mbfb_v \mbfv_k - \mbf{g}(\check{\mbfx}_k, \mbfu_k)$. 
Furthermore, $\dx_k = \mbfx_k - \check{\mbfx}_k$ 
and $\check{\mbfy}$ is the predicted measurement. 
%
The \textit{a posteriori} state estimate $\hat{\mbfx}_k$ defined as
\begin{align}
    \hat{\mbfx}_k &= \check{\mbfx}_k + \mbfk (\mbfy_k - \check{\mbfy}_k) , \label{eq:update}
\end{align}
where the Kalman gain $\mbfk$ is assumed to be available. 
%
%
For the purpose of having more compact equations, the \textit{a posteriori} state error is defined as
$\hat{\mbfe}_k = \mbfx_k - \hat{\mbfx}_k$ where $\mbfx_k$ is the true state. 
The \textit{a posteriori} state estimate error is
    \begin{align}
        \hat{\mbfe}_k &= \mbfx_k - \hat{\mbfx}_k \\
        &= \mbfx_k - ( \check{\mbfx}_k + \mbfk (\mbfy_k - \check{\mbfy}_k) ) \\
        &= \dx_k - \mbf{K} \dy_k. 
        \label{eq:error_apos}
    \end{align}
when combined with \eqref{eq:update}. The \textit{a posteriori} covariance is thus $\mbfpp_k^+ = \ex{\hat{\mbfe}_k \hat{\mbfe}_k^\trans}$.


We assume these nonlinear equations are static and has no memory so that quadratic constraints may be used to bound the input-output relationship. 
A large class of quadratic constraints \cite{ref_21_megretski1997system} are available to choose from. 
For an input $\mbf{b} \subset \mathbb{B} \in \mathbb{R}^{n_b} $ and output $\mbf{d} \in \mathbb{R}^{n_d}$, the quadratic constraint satisfies the inequality equation of 
\begin{align}
    \bbm \mbf{b} \\ \mbf{d} \ebm^\trans
    \mbs{\Lambda}
    \bbm \mbf{b} \\ \mbf{d} \ebm
    \geq 0,
    \label{eq:QC_general0}
\end{align}
where the entries of $\mbs{\Lambda}$ define the type of quadratic constraint. For example, a norm bound quadratic constraint that satisfies $\norm{\mbf{d}}_2 \leq \gamma \norm{\mbf{b}}_2$, $\gamma>0, \gamma \in \mathbb{R}$ has $\mbs{\Lambda} = \bbm \gamma^2 \bone & \bzero \\ \bzero & -\bone \ebm$. 
%
We may further extend the example by applying a quadratic constraint inequality \eqref{eq:QC_general0} to the nonlinear equation $\grave{\mbs{\Delta}}$ in \eqref{eq:eom_dt_2} that has $\mbs{\theta}_k$ as input and $\mbf{p}_k$ as output. 
By defining $\mbf{U} = \mathrm{diag}(\mbf{C}_\theta, \bone)$, then $\bbm \mbs{\theta}_k & \mbf{p}_k \ebm^\trans = \mbf{U} \bbm \delta \mbf{x}_k & \mbf{p}_k \ebm^\trans$. The quadratic constraint can be expressed as
\begin{align}
    \bbm \delta \mbf{x}_k \\ \mbf{p}_k \ebm^\trans
    \mbf{U}^\trans \mbs{\Lambda} \mbf{U}
    \bbm \delta \mbf{x}_k \\ \mbf{p}_k \ebm 
    &\geq 0 ,
    \\
    \bbm \delta \mbf{x}_k \\ \mbf{p}_k \ebm^\trans
    \mbf{M}
    \bbm \delta \mbf{x}_k \\ \mbf{p}_k \ebm 
    &\geq 0,
    \label{eq:QC_general}
\end{align}
where $\mbf{M} = \mbf{U}^\trans \mbs{\Lambda}\mbf{U}$. The same steps may be used for the nonlinear equation $\acute{\mbs{\Delta}}_k$ in \eqref{eq:delta}. 
In summary, the quadratic constraint inequality bounds the input-output relationship of the nonlinear function $\mbs{\Delta}$. 

Many nonlinear equations, such as polynomials, cannot be globally bounded with quadratic constraints. 
In this instance, the quadratic constraint is only able to capture the nonlinearities locally in a subset $\mathbb{B}$ that is predefined by the user as a tuning knob as to how much nonlinearities are to be considered.

\section{Bounding of Covariance for Nonlinear Equations}
\label{section:covariance_bounding}

This section provides the main results on bounding the estimated state error covariance for both the process equation and measurement equation 

\begin{theorem}
    \label{theorem:dt_covariance}
Consider the time update equation described by \eqref{eq:eom_dt_1}-\eqref{eq:eom_dt_last} with white process noise satisfying 
$\mathbb{E}(\mbf{w}_k\mbf{w}_k^\trans)=\mbf{Q}$, 
an estimation error covariance at time step $k$ of 
$\mathbb{E} \left(\delta \mbf{x}_k \delta \mbf{x}_k^\trans\right)=\mbf{P}^+_k$  and 
at $k+1$ of 
$\mathbb{E} \left(\delta \mbf{x}_{k+1} \delta \mbf{x}_{k+1}^\trans \right)=\mbf{P}^-_{k+1}$.
Furthermore, $\mbf{p}_k=\mbs{\Delta}(\mbs{\theta}_k)$ satisfies \eqref{eq:QC_general}. 
If there exist $\mbf{Z} \in \mathbb{S}^n$, $\mbf{Y} \in \mathbb{S}^n$ and $\xi\geq 0$ satisfying
\begin{align}
    \mbf{Z}-\mbf{Y} &< 0,  \label{eq:LMI1_dt0}
    \\
    \bbm 
        \mbf{A}_k^\trans (\mbf{Z}-\mbf{Y}) \mbf{A}_k  & \mbf{A}_k^\trans \mbf{ZB}_p \\
        \mbf{B}_p^\trans \mbf{Z} \mbf{A}_k & \mbf{B}_p^\trans \mbf{Z} \mbf{B}_p
    \ebm +
    \xi \mbf{M}
    &\leq 0 ,
    \label{eq:LMI1_dt}
\end{align}
then 
\begin{align}
    \trace(\mbf{P}_{k+1}^- \mbf{Z})
    &\leq 
    \trace( \mbf{B}_w \mbf{Q} \mbf{B}_w^\trans \mbf{Z} ) + 
    \trace{(\mbf{A}_k \mbf{P}^+_k \mbf{A}_k^\trans \mbf{Y})}.
    \label{eq:claim_dt}
\end{align}
\end{theorem}

\begin{proof}
    Let $\mbf{K} \in \mathbb{S}^n$ be a free variable and $\mbf{r} \in \mbbRv{n}$. The quadratic function $f(\mbf{K}, \mbf{r})=\mbf{r}^\trans \mbf{K} \mbf{r}$ is evaluated at consecutive time step 
\begin{align}
    f&(\mbf{Z}, \delta \mbf{x}_{k+1}) 
    - f(\mbf{Y}, \mbf{A}_k \delta \mbf{x}_k) =  
    \delta \mbf{x}_{k+1}^\trans \mbf{Z} \delta \mbf{x}_{k+1} - 
    \delta \mbf{x}_k^\trans  \mbf{A}_k^\trans \mbf{Y} \mbf{A}_k \delta \mbf{x}_k 
    \\
    &=
    (\mbf{A}_k \delta \mbf{x}_k + \mbf{B}_p \mbf{p}_k + \mbf{B}_w \mbf{w}_k)^\trans 
    \mbf{Z}
    (\mbf{A}_k \delta \mbf{x}_k + \mbf{B}_p \mbf{p}_k + \mbf{B}_w \mbf{w}_k)
    - \delta \mbf{x}_k^\trans \mbf{A}_k^\trans \mbf{Y} \mbf{A}_k \delta  \mbf{x}_k 
    \\
    &= 
    \delta \mbf{x}_k^\trans \left( \mbf{A}_k^\trans (\mbf{Z}-\mbf{Y}) \mbf{A}_k \right) \delta \mbf{x}_k 
    + \mbf{p}_k^\trans \mbf{B}_p^\trans \mbf{Z} \mbf{B}_p \mbf{p}_k
    + \mbf{w}_k^\trans \mbf{B}_w^\trans \mbf{Z} \mbf{B}_w \mbf{w}_k + \notag \\
    & \qquad \qquad
    2 \delta \mbf{x}_k^\trans \mbf{A}_k^\trans \mbf{Z} \mbf{B}_p \mbf{p}_k
    + 2 \delta \mbf{x}_k^\trans \mbf{A}_k^\trans \mbf{Z} \mbf{B}_w \mbf{w}_k 
    + 2 \mbf{p}_k^\trans \mbf{B}_p^\trans \mbf{Z} \mbf{B}_w \mbf{w}_k .
    \label{eq:Lyap1}
\end{align}
Pre- and post-multiplying \eqref{eq:LMI1_dt} by $\bbm \delta \mbf{x}_k^\trans & \mbf{p}_k^\trans \ebm^\trans$ and $\bbm \delta \mbf{x}_k^\trans & \mbf{p}_k^\trans \ebm$ respectively, yields
\begin{align}
    \delta \mbf{x}_k^\trans \left(\mbf{A}_k^\trans (\mbf{Z}-\mbf{Y}) \mbf{A}_k \right) \delta \mbf{x}_k 
    + 2 \delta \mbf{x}_k^\trans \mbf{A}_k^\trans \mbf{Z} \mbf{B}_p \mbf{p}_k
    + \mbf{p}_k^\trans \mbf{B}_p^\trans \mbf{Z} \mbf{B}_p \mbf{p}_k
    +
    \xi_r
    \bbm \delta \mbf{x}_k \\ \mbf{p}_k \ebm^\trans
    \mbf{M}
    \bbm \delta \mbf{x}_k \\ \mbf{p}_k \ebm 
    & \leq 0.
    \label{eq:Lyap2}
\end{align}
Recognizing that \eqref{eq:QC_general} holds, the last sum term of \eqref{eq:Lyap2} is positive which implies its absence from the inequality still infers \eqref{eq:Lyap2} (S-procedure~\cite[p.~23]{ref_26_boyd1994LMIcontrol}).
The inequality of \eqref{eq:Lyap2} with the last term removed is substituted into \eqref{eq:Lyap1} to yield
\begin{align}
    \delta \mbf{x}_{k+1}^\trans \mbf{Z} \delta \mbf{x}_{k+1} - 
    \delta \mbf{x}_k^\trans \mbf{A}_k^\trans \mbf{Y} \mbf{A}_k \delta \mbf{x}_k
    &\leq 
    \mbf{w}_k^\trans \mbf{B}_w^\trans \mbf{Z} \mbf{B}_w \mbf{w}_k 
    + 2 \delta \mbf{x}_k^\trans \mbf{A}_k^\trans \mbf{Z} \mbf{B}_w \mbf{w}_k 
    + 2 \mbf{p}_k^\trans \mbf{B}_p^\trans \mbf{Z} \mbf{B}_w \mbf{w}_k .
\end{align}

Taking the trace and then the expection of both sides yields 
\begin{align}
    \mathbb{E}(\trace (\delta \mbf{x}_{k+1}^\trans \mbf{Z} \delta \mbf{x}_{k+1})) - 
    \mathbb{E}(\trace(\delta \mbf{x}_k^\trans \mbf{A}_k^\trans \mbf{Y} \mbf{A}_k \delta \mbf{x}_k))
    &\leq 
    \notag \\
    \mathbb{E}(\trace(\mbf{w}_k^\trans \mbf{B}_w^\trans \mbf{Z} \mbf{B}_w \mbf{w}_k ))
    + \mathbb{E}(\trace(2 \delta \mbf{x}_k^\trans \mbf{A}_k^\trans \mbf{Z} \mbf{B}_w \mbf{w}_k ))
    &+ \mathbb{E}(\trace(2 \mbf{p}_k^\trans \mbf{B}_p^\trans \mbf{Z} \mbf{B}_w \mbf{w}_k )).
\end{align}
Applying the cyclic property of the trace to the terms of this inequality results in 
\begin{align}    
    \mathbb{E}(\trace(\delta \mbf{x}_{k+1} \delta \mbf{x}_{k+1}^\trans \mbf{Z} ) )- 
    \mathbb{E}(\trace(\mbf{A}_k \delta \mbf{x}_k \delta \mbf{x}_k^\trans \mbf{A}_k^\trans \mbf{Y} ))
    &\leq 
    \notag \\
    \mathbb{E}(\trace(\mbf{B}_w \mbf{w}_k \mbf{w}_k^\trans \mbf{B}_w^\trans \mbf{Z}))
    + 2 \mathbb{E}(\trace(\mbf{w}_k \delta \mbf{x}_k^\trans \mbf{A}^\trans \mbf{Z} \mbf{B}_w))
    &+ 2 \mathbb{E}(\trace(\mbf{w}_k \mbf{p}_k^\trans \mbf{B}_p^\trans \mbf{Z} \mbf{B}_w )).
\end{align}

The trace is a linear operator, thus the expectation operator may be brought inside the trace to yield
\begin{align}    
    \trace(\mathbb{E}(\delta \mbf{x}_{k+1} \delta \mbf{x}_{k+1}^\trans) \mbf{Z}  )- 
    \trace(\mbf{A}_k \mathbb{E}(\delta \mbf{x}_k \delta \mbf{x}_k^\trans) \mbf{A}_k^\trans \mbf{Y} )
    &\leq 
    \notag \\
    \trace( \mbf{B}_w \mathbb{E}( \mbf{w}_k \mbf{w}_k^\trans) \mbf{B}_w^\trans \mbf{Z})
    + 2 \trace( \mathbb{E}(\mbf{w}_k \delta \mbf{x}_k^\trans) \mbf{A}^\trans \mbf{Z} \mbf{B}_w)
    &+ 2 \trace( \mathbb{E}(\mbf{w}_k \mbf{p}_k^\trans) \mbf{B}_p^\trans \mbf{Z} \mbf{B}_w ).
    \label{eq:dt_proof_interm}
\end{align}

Finally, the covariance of the process noise is defined as $\mbf{Q}=\mathbb{E}(\mbf{w}_k \mbf{w}_k^\trans)$, the estimation error covariances at time steps $k+1$ and $k$ are respectively $\mbf{P}^-_{k+1} = \mathbb{E} (\delta \mbf{x}_{k+1} \delta \mbf{x}_{k+1}^\trans)$ and $\mbf{P}^+_k = \mathbb{E}(\delta \mbf{x}_k \delta \mbf{x}_k^\trans)$. Additionally, the last two terms involve uncorrelated random processes, therefore $\mathbb{E}(\mbf{w}_k \delta \mbf{x}_k^\trans)=0$ and $\mathbb{E}(\mbf{w}_k \mbf{p}_k^\trans)=0$. 
Substituting these terms into \eqref{eq:dt_proof_interm} completes the proof.
\end{proof}

The equations \eqref{eq:LMI1_dt0} and \eqref{eq:LMI1_dt} are Linear Matrix Inequality (LMI) equations with respect to $\mbf{Y}$ and $\mbf{Z}$. Either variables or both variables may be used as free variables within a SDP. 
In the case when the system is linear, both \eqref{eq:LMI1_dt0} and \eqref{eq:LMI1_dt} are not needed to account for any nonlinearities, therefore
it can be shown that both variables can be set to $\mbf{Y}=\mbf{Z}=\bone$ in \eqref{eq:claim_dt}, this returns the Stein equation that's inside the trace operator. 

A key innovation in Theorem~\ref{theorem:dt_covariance} is the inclusion of two 
matrix variables $\mbf{Z}$ and $\mbf{Y}$ to form $\mbf{A}_k^\trans (\mbf{Z}-\mbf{Y}) \mbf{A}_k$ in the first entry of \eqref{eq:LMI1_dt}. There are no restrictions to the definiteness of 
$\mbf{Z}$ and $\mbf{Y}$ individually, rather there is only the relative definiteness constraint $\mbf{Z}-\mbf{Y} <0$. 
This is in contrast to typical discrete-time LMI formulations that feature a term analogous to 
$\mbf{A}_k^\trans \mbf{Z} \mbf{A}_k-\mbf{Z}<0, \mbf{Z}>0$ \cite[p.~97]{duan2013lmis} that implies feasibility only when the absolute value of all eigenvalues of $\mbf{A}_k$ are less than $1$, or equivalently, the discrete-time system $\mbf{x}_{k+1}=\mbf{A}_k \mbf{x}_k$ is asymptotically stable. The proposed Theorem~\ref{theorem:dt_covariance} is therefore capable of handling unstable discrete-time systems, whereas 
similar LMIs based on the stochastic interpretation of the $\mathcal{H}_2$ norm are only applicable to asymptotically stable discrete-time systems.

Bounding of the covariance in the measurement update follows similar steps to the process covariance update when applied on the \textit{a posteriori} state estimate error. This is shown in the following theorem. 

\begin{theorem}
\label{theorem:meas_bound}
    Considering the measurement update equation of \eqref{eq:meas_comp}-\eqref{eq:meas_end} where measurement white noise satisfying $\mbf{R}=\ex{\mbf{v}_k \mbf{v}_k^\trans}$, a priori estimation error covariance at time step $k$ as $\mbfpp_{k}^- = \ex{\dx_k \dx_k^\trans}$. Furthermore, $\mbs{\rho}_k = \mbs{\Delta}(\mbs{\mu}_k)$ satisfies the quadratic constraint \eqref{eq:QC_general}. If there exists $\mbfz \in \mathbb{S}^n$, $\mbfyy \in \mathbb{S}^n$ and $\xi \in \mathbb{R}_{\geq0}$ satisfying
    \begin{align}
        \bbm
           -\mbfyy& \mbf{H}_k^\trans \mbfk^\trans \mbfz \mbfk \mbfb_\rho - \mbfz \mbfk \mbfb_\rho \\
           * & \mbfb_\rho^\trans \mbfk^\trans \mbfz \mbfk \mbfb_\rho
        \ebm
        + \xi \mbf{M}
        &\leq 0,
        \label{eq:theorem}
    \end{align}
    then
    \begin{align}
        \trace(\mbfpp^+_k \mbfz) &\leq 
        \trace \left( 
            \mbfpp^-_k 
            (\bone - \mbfk \mbfh)^\trans \mbfz (\bone - \mbfk \mbfh)
        \right)
        + 
        \trace ( \mbfk \mbfb_v \mbf{R} \mbfb_v^\trans \mbfk^\trans \mbfz)
        + 
        \trace ( \mbfpp_k^- \mbfyy) .
        \label{eq:result}
    \end{align}
\end{theorem}
\begin{proof}
    
    Let $\mbfxx \in \mathbb{S}^n$ be a free variable and $\mbf{r} \in \mathbb{R}^n$. The quadratic function $\tau(\mbfxx, \mbf{r}) = \mbf{r}^\trans \mbfxx \mbf{r}$ is evaluated as $\tau(\mbfz, \hat{\mbfe}_k)$ with substitution of \eqref{eq:error_apos} as
    \begin{align}
        \hat{\mbfe}_k^\trans \mbfz \hat{\mbfe}_k 
        &= (\dx_k - \mbfk \dy_k)^\trans \mbfz (\dx_k - \mbfk \dy_k)
        \\
        &= 
        \dx_k^\trans \mbfz \dx_k - \dx_k^\trans \mbfz \mbfk \dy_k - \dy_k^\trans \mbfk^\trans \mbfz \dx_k
        + \dy_k^\trans \mbfk^\trans \mbfz \mbfk \dy_k .       
        \label{eq:quad_first}
    \end{align}
    Substituting \eqref{eq:meas_comp} into \eqref{eq:quad_first} yields
    \begin{align}
        \hat{\mbfe}_k^\trans \mbfz \hat{\mbfe}_k 
        &=
        \dx_k^\trans \mbfz \dx_k - 
        \dx_k^\trans \mbfz \mbfk (\mbfh \dx_k + \mbfb_\rho \mbs{\rho}_k + \mbfb_v \mbfv_k) 
        \notag \\
        & \qquad -(\mbfh \dx_k + \mbfb_\rho \mbs{\rho}_k + \mbfb_v \mbfv_k)^\trans \mbfk^\trans \mbfz \dx_k
        \notag \\ & \qquad 
        + (\mbfh \dx_k + \mbfb_\rho \mbs{\rho}_k + \mbfb_v \mbfv_k)^\trans \mbfk^\trans \mbfz \mbfk (\mbfh \dx_k + \mbfb_\rho \mbs{\rho}_k + \mbfb_v \mbfv_k)
        \\
        &= 
        \dx_k^\trans  \left( \mbfz - \mbfz \mbfk \mbfh - \mbfh^\trans \mbfk^\trans \mbfz + \mbfh^\trans \mbfk^\trans \mbfz \mbfk \mbfh \right) \dx_k
        \notag \\ & \qquad  
        -\dx_k^\trans \mbfz \mbfk \mbfb_\rho \mbs{\rho}_k 
        -\dx_k^\trans \mbfz \mbfk \mbfb_v \mbfv_k
        - \mbs{\rho}_k^\trans \mbfb_\rho^\trans \mbfk^\trans \mbfz \dx_k
        - \mbfv_k^\trans \mbfb_v^\trans \mbfk^\trans \mbfz \dx_k
        \notag \\ & \qquad 
        + \dx^\trans_k \mbfh^\trans  \mbfk^\trans \mbfz \mbfk  \mbfb_\rho \mbfp_p 
        + \dx^\trans_k \mbfh^\trans  \mbfk^\trans \mbfz \mbfk \mbfb_v \mbfv_k
        \notag \\ & \qquad 
        + \mbs{\rho}_k^\trans \mbfb_\rho^\trans  \mbfk^\trans \mbfz \mbfk  \mbfh \dx_k
        + \mbs{\rho}_k^\trans \mbfb_\rho^\trans  \mbfk^\trans \mbfz \mbfk  \mbfb_\rho \mbfp_p 
        + \mbs{\rho}_k^\trans \mbfb_\rho^\trans  \mbfk^\trans \mbfz \mbfk  \mbfb_v \mbfv_k
        \notag \\ & \qquad 
        + \mbfv_k^\trans \mbfb_v^\trans  \mbfk^\trans \mbfz \mbfk  \mbfh \dx_k
        + \mbfv_k^\trans \mbfb_v^\trans  \mbfk^\trans \mbfz \mbfk  \mbfb_\rho \mbfp_p 
        + \mbfv_k^\trans \mbfb_v^\trans  \mbfk^\trans \mbfz \mbfk  \mbfb_v \mbfv_k .\label{eq:quad_second}
    \end{align}

    Multiplying \eqref{eq:theorem} left and right by $\bbm \dx_k^\trans & \mbs{\rho}^\trans_k \ebm$
    and $\bbm \dx_k^\trans & \mbs{\rho}^\trans_k \ebm^\trans$ on both sides respectively yields
    \begin{align}
        \dx^\trans_k \mbfh^\trans  \mbfk^\trans \mbfz \mbfk  \mbfb_\rho \mbfp_p 
        -\dx_k^\trans \mbfz \mbfk \mbfb_\rho \mbs{\rho}_k 
        + \mbs{\rho}_k^\trans \mbfb_\rho^\trans  \mbfk^\trans \mbfz \mbfk  \mbfh \dx_k
        - \mbs{\rho}_k^\trans \mbfb_\rho^\trans \mbfk^\trans  \mbfz & \dx_k
        \notag \\
        + \mbs{\rho}_k^\trans \mbfb_\rho^\trans  \mbfk^\trans \mbfz \mbfk  \mbfb_\rho \mbfp_p 
        + \xi \left( \bbm \dx_k^\trans & \mbs{\rho}^\trans_k \ebm \mbf{M} \bbm \dx_k^\trans & \mbs{\rho}^\trans_k \ebm^\trans \right)
        \leq & \dx_k^\trans \mbfyy \dx_k . \label{eq:theorem2_expand}
    \end{align}
    Recognizing that \eqref{eq:QC_general} holds, the last sum term on the left side of \eqref{eq:theorem2_expand} is positive semidefinite which implies its absence from the inequality still infers \eqref{eq:theorem2_expand}. 
    The inequality of \eqref{eq:theorem2_expand} with the last term removed is substituted into \eqref{eq:quad_second} to yield 
\begin{align}
    \hat{\mbfe}_k^\trans \mbfz \hat{\mbfe}_k 
        &\leq
        \dx_k^\trans  \left( \mbfz - \mbfz \mbfk \mbfh - \mbfh^\trans \mbfk^\trans \mbfz + \mbfh^\trans \mbfk^\trans \mbfz \mbfk \mbfh \right) \dx_k
        \notag \\ & \qquad  
        -\dx_k^\trans \mbfz \mbfk \mbfb_v \mbfv_k
        - \mbfv_k^\trans \mbfb_v^\trans \mbfk^\trans \mbfz \dx_k
        + \dx^\trans_k \mbfh^\trans  \mbfk^\trans \mbfz \mbfk \mbfb_v \mbfv_k
        + \mbs{\rho}_k^\trans \mbfb_\rho^\trans  \mbfk^\trans \mbfz \mbfk  \mbfb_v \mbfv_k 
        \notag \\ & \qquad 
        + \mbfv_k^\trans \mbfb_v^\trans  \mbfk^\trans \mbfz \mbfk  \mbfh \dx_k
        + \mbfv_k^\trans \mbfb_v^\trans  \mbfk^\trans \mbfz \mbfk  \mbfb_\rho \mbfp_p 
        + \mbfv_k^\trans \mbfb_v^\trans  \mbfk^\trans \mbfz \mbfk  \mbfb_v \mbfv_k 
        + \dx_k^\trans \mbfyy \dx_k .
\end{align}

Taking the trace and then the expectation of both sides yield
\begin{align}
    \mathbb{E} ( \trace ( &\hat{\mbfe}_k^\trans \mbfz \hat{\mbfe}_k ) )
        \leq
        \mathbb{E}\left( \right. \trace (
        \dx_k^\trans  \left( \mbfz - \mbfz \mbfk \mbfh - \mbfh^\trans \mbfk^\trans \mbfz + \mbfh^\trans \mbfk^\trans \mbfz \mbfk \mbfh \right) \dx_k 
        \notag \\ &  
        -\dx_k^\trans \mbfz \mbfk \mbfb_v \mbfv_k
        - \mbfv_k^\trans \mbfb_v^\trans \mbfk^\trans \mbfz \dx_k
        + \dx^\trans_k \mbfh^\trans  \mbfk^\trans \mbfz \mbfk \mbfb_v \mbfv_k
        + \mbs{\rho}_k^\trans \mbfb_\rho^\trans  \mbfk^\trans \mbfz \mbfk  \mbfb_v \mbfv_k 
        \notag \\ &  
        + \mbfv_k^\trans \mbfb_v^\trans  \mbfk^\trans \mbfz \mbfk  \mbfh \dx_k
        + \mbfv_k^\trans \mbfb_v^\trans  \mbfk^\trans \mbfz \mbfk  \mbfb_\rho \mbfp_p 
        + \mbfv_k^\trans \mbfb_v^\trans  \mbfk^\trans \mbfz \mbfk  \mbfb_v \mbfv_k 
        + \dx_k^\trans \mbfyy \dx_k ) \left. \right).
\end{align}
Applying the cyclic property of the trace to the terms of this inequality yields
\begin{align}
    \mathbb{E} ( \trace ( & \hat{\mbfe}_k \hat{\mbfe}_k^\trans \mbfz  ) )
        \leq
        \mathbb{E}\left( \right. \trace (
        \dx_k \dx_k^\trans  \left( \mbfz - \mbfz \mbfk \mbfh - \mbfh^\trans \mbfk^\trans \mbfz + \mbfh^\trans \mbfk^\trans \mbfz \mbfk \mbfh \right)  
        \notag \\ &  
        - \mbfv_k \dx_k^\trans \mbfz \mbfk \mbfb_v 
        - \dx_k \mbfv_k^\trans \mbfb_v^\trans \mbfk^\trans \mbfz 
        + \mbfv_k \dx^\trans_k \mbfh^\trans  \mbfk^\trans \mbfz \mbfk \mbfb_v 
        + \mbfv_k \mbs{\rho}_k^\trans \mbfb_\rho^\trans  \mbfk^\trans \mbfz \mbfk  \mbfb_v  
        \notag \\ &  
        + \dx_k \mbfv_k^\trans \mbfb_v^\trans  \mbfk^\trans \mbfz \mbfk  \mbfh 
        + \mbfp_p  \mbfv_k^\trans \mbfb_v^\trans  \mbfk^\trans \mbfz \mbfk  \mbfb_\rho 
        + \mbfk  \mbfb_v \mbfv_k  \mbfv_k^\trans \mbfb_v^\trans  \mbfk^\trans \mbfz  
        + \dx_k \dx_k^\trans \mbfyy  ) \left. \right).
\end{align}
Since the trace is a linear operator, the expectation operator is brought inside the trace to yield
\begin{align}
    \trace ( \mathbb{E} (  &\hat{\mbfe}_k \hat{\mbfe}_k^\trans ) \mbfz   )
        \leq
        \trace\left( \right.  \mathbb{E}(
        \dx_k \dx_k^\trans)  \left( \mbfz - \mbfz \mbfk \mbfh - \mbfh^\trans \mbfk^\trans \mbfz + \mbfh^\trans \mbfk^\trans \mbfz \mbfk \mbfh \right)  
        \notag \\ & \qquad 
        - \mathbb{E}(\mbfv_k \dx_k^\trans) \mbfz \mbfk \mbfb_v 
        - \mathbb{E}(\dx_k \mbfv_k^\trans) \mbfb_v^\trans \mbfk^\trans \mbfz 
        + \mathbb{E}(\mbfv_k \dx^\trans_k) \mbfh^\trans  \mbfk^\trans \mbfz \mbfk \mbfb_v 
        \notag \\ & \qquad
        + \mathbb{E}(\mbfv_k \mbs{\rho}_k^\trans) \mbfb_\rho^\trans  \mbfk^\trans \mbfz \mbfk  \mbfb_v  
        + \mathbb{E}(\dx_k \mbfv_k^\trans) \mbfb_v^\trans  \mbfk^\trans \mbfz \mbfk  \mbfh 
        + \mathbb{E}(\mbfp_p  \mbfv_k^\trans) \mbfb_v^\trans  \mbfk^\trans \mbfz \mbfk  \mbfb_\rho 
        \notag \\ & \qquad
        + \mbfk  \mbfb_v  \mathbb{E}(\mbfv_k  \mbfv_k^\trans) \mbfb_v^\trans  \mbfk^\trans \mbfz 
        + \mathbb{E}(\dx_k \dx_k^\trans) \mbfyy   \left. \right).
\end{align}
A more compact term of 
$(\mbfz - \mbfz \mbfk \mbfh - \mbfh^\trans \mbfk^\trans \mbfz + \mbfh^\trans \mbfk^\trans \mbfz \mbfk \mbfh ) = (\bone - \mbfk \mbfh)^\trans \mbfz (\bone - \mbfk \mbfh)$ is used to reflect the similarities to the Kalman filter. 
Substituting $\mbfpp_k^+ = \ex{\hat{\mbfe}_k \hat{\mbfe}_k^\trans}$, $\mbfpp_{k}^- = \ex{\dx_k \dx_k^\trans}$ for the state error covariances and $\mbf{R}=\ex{\mbf{w}_k \mbf{w}_k^\trans}$. 
Recognizing that the measurement noise $\mbfv_k$ is uncorrelated to the states $\dx_k$ and $\mbfp$ allows substitution of 
$\ex{\dx_k^\trans \mbfv_k } = \ex{\mbfv_k \dx_k^\trans}=\bzero$, $\ex{\mbfv_k \mbfp_p^\trans }=\ex{\mbfp_p  \mbfv_k^\trans }=\bzero$ yields
    \begin{align}
        \trace(\mbfpp^+_k \mbfz) &\leq 
        \trace \left( 
            \mbfpp^-_k 
            (\bone - \mbfk \mbfh)^\trans \mbfz (\bone - \mbfk \mbfh)
        \right)
        +
        \trace ( \mbfk \mbfb_v \mbf{R} \mbfb_v^\trans \mbfk^\trans \mbfz)
        + 
        \trace ( \mbfpp_k^- \mbfyy ),
        \label{eq:result_rep}
    \end{align}
which completes the proof. 
\end{proof}

It is notable that for the linear case, it can be shown that \eqref{eq:result} is an equality with $\mbfyy=\bzero$ and $\mbf{Z}=\bone$. 

\section{Description of Algorithm}
\label{section:algorithm}


The traditional EKF notably has two steps. In the first step, the time-update equation for covariance from $\mbfpp_k^+$  to $\mbfpp_{k+1}^-$ is $\mbfpp_{k+1}^- = \mbf{A}_k \mbfpp_k^+ \mbf{A}_k^\trans + \mbf{B}_w \mbf{Q} \mbf{B}_w^\trans$. 
The proposed algorithm that takes its place is 
$\trace(\mbf{P}^-_{k+1} \mbf{Z})    =    \trace( \mbf{B}_w \mbf{Q} \mbf{B}_w^\trans \mbf{Z} ) + 
\trace{(\mbf{A}_k \mbf{P}^+_k \mbf{A}_k^\trans \mbf{Y}^*)}$ with $\mbfyy^*$ being the optimal solution of the following SDP
\begin{align}
    \min_{\xi \in \mathbb{R} , \mbfyy \in \mathbb{S}^n} \quad & 
        \trace{(\mbf{A}_k \mbf{P}^+_k \mbf{A}_k^\trans \mbf{Y})}  \label{eq:sdp_eq_first}
        \\
    \textrm{s.t.} 
    \quad & 
    \mbf{Z}-\mbf{Y} \leq 0,  
    \\ \quad & 
    \xi \geq 0, 
    \\ \quad & 
    \bbm 
        \mbf{A}_k^\trans (\mbf{Z}-\mbf{Y}) \mbf{A}_k  & \mbf{A}_k^\trans \mbf{ZB}_p \\
        \mbf{B}_p^\trans \mbf{Z} \mbf{A}_k & \mbf{B}_p^\trans \mbf{Z} \mbf{B}_p
    \ebm +
    \xi \mbf{M}
    \leq 0. \label{eq:sdp_eq_last}
\end{align}
where $\mbfz$ is chosen such that $\mbfz = \onehalf \bone_i \bone_j^\trans + \onehalf \bone_j \bone_i^\trans$ to solve for the $i^{th}$ row and $j^{th}$ column entry of $\mbfpp_{k+1}^-$. 


The second step in a traditional EKF is the measurement update $\mbfpp^+_{k+1} = (\bone - \mbf{K}_{k+1} \mbf{H}) \mbfpp_{k+1}^-$. The proposed algorithm that takes its place is $\trace(\mbfpp^+_{k+1} \mbfz) =\trace \left( \mbfpp^-_{k+1} (\bone - \mbf{K}_{k+1} \mbf{H}_{k+1})^\trans \mbfz (\bone - \mbf{K}_{k+1} \mbf{H}_{k+1}) \right)
        + \trace ( \mbf{K}_{k+1} \mbfb_v \mbf{R} \mbfb_v^\trans \mbf{K}_{k+1}^\trans \mbfz )
        + \trace ( \mbfpp_{k+1}^- \mbfyy^* )$ 
with $\mbf{Y}^*$ being the optimal solution of the following SDP
\begin{align}
    \min_{\xi \in \mathbb{R} , \mbfyy \in \mathbb{S}^n} \quad & 
    \trace 
        \trace ( \mbfpp_{k+1}^- \mbfyy )
    \\
    \textrm{s.t.} \quad & \bbm
           -\mbfyy & \mbf{H}_{k+1}^\trans \mbf{K}_{k+1}^\trans \mbfz \mbf{K}_{k+1} \mbfb_\rho - \mbfz \mbf{K}_{k+1} \mbfb_\rho \\
           * & \mbfb_\rho^\trans \mbf{K}_{k+1}^\trans \mbfz \mbf{K}_{k+1} \mbfb_\rho
        \ebm
        + \xi \mbf{M}
        \leq 0
        \\
        &\xi \geq 0.
\end{align}
where $\mbfz$ is chosen such that $\mbfz = \onehalf \bone_i \bone_j^\trans + \onehalf \bone_j \bone_i^\trans$ to solve for the $i^{th}$ row and $j^{th}$ column entry of $\mbfpp_{k+1}^+$. 

It is noteworthy that the lower bounds of each entry of the covariance matrix for both algorithms may be obtained by instead choosing $\mbfz = -( \onehalf \bone_i \bone_j^\trans + \onehalf \bone_j \bone_i^\trans)$. Similarly, if the lower and upper bounds of the $\trace(\mbfpp_{k+1}^-)$ or $\trace(\mbfpp_{k+1}^+)$ are desired, choosing $\mbfz=-\bone$ and $\mbfz-\bone$ can be used respectively in both algorithms to obtain the necessary bounds. 

\section{Conclusions}
In this paper, the challenge of estimating conservative state error covariance for nonlinear systems within the context of the extended Kalman filter is investigated. 
Using quadratic constraints to bound the nonlinearities locally, the upper bounds of individual elements within the state error covariance matrix can be obtained by solving semidefinite programs.

\section*{Acknowledgments}
Suggestions and encouragement by Dr Ryan J Caverly and Dr Demoz Gebre-Egziabher is greatly appreciated by the authors. 
Sze Kwan Cheah contributions were partially supported by a University of Minnesota Informatics Institute MnDRIVE Ph.D. Graduate Assistantship.

\bibliography{biblio} 

\end{document}